\documentclass[12pt,leqno]{amsart}
\usepackage{amssymb,enumerate}
\newcommand{\FF}{{\mathbb{F}}}

\newcommand{\GL}{{\operatorname{GL}}}

\newcommand{\PSL}{{\operatorname{L}}}

\newcommand{\Sp}{{\operatorname{Sp}}}

\newcommand{\EE}{{\operatorname{End}}}
\newcommand{\Hom}{{\operatorname{Hom}}}
 
\newtheorem{thm}{Theorem}[section]
\newtheorem{theorem}[thm]{Theorem}
\newtheorem{lemma}[thm]{Lemma}
\newtheorem{cor}[thm]{Corollary}

\newtheorem{question}[thm]{Question}

\theoremstyle{definition}

\theoremstyle{remark}

\begin{document}

\title{Adequate Subgroups II}


\author{Robert Guralnick}
\address{3620 S. Vermont Ave, Department of Mathematics, University of
  Southern California, Los Angeles, CA 90089-2532, USA, phone:  213-740-3787, fax: 213-740-2424}
\makeatletter\email{guralnic@usc.edu} 


 
\begin{abstract}  
The notion of adequate subgroups was introduced by Jack Thorne \cite{T}.
It is a weakening of the notion of big subgroup used by Wiles and Taylor 
in proving   automorphy lifting theorems for certain Galois representations.  
Using this idea, Thorne was able to prove some new lifting theorems. 
 It was shown in \cite{GHTT} that
certain groups were adequate.   One of the key aspects was the question
of whether the span of the semsimple elements in the group is the full endomorphism
ring of an absolutely irreducible module.  We show that this is the case in 
prime characteristic $p$  for
$p$-solvable groups as long the dimension is not divisible by $p$.  We also
observe that the condition holds for certain infinite groups.  Finally, we present
the first examples  showing that this condition need not hold and give a negative
answer to a question of Richard Taylor.
 \end{abstract}

\maketitle


\section{Introduction} \label{sec:intro}

Let $k$ be a field of characteristic $p$ and let $V$ be a finite dimensional vector space over $k$.
Let $\rho:G \rightarrow  \GL(V)$ be an absolutely irreducible representation.
Following \cite{T}, we say $(G,V)$ is {\it adequate}  if the following conditions hold
(we rephrase the conditions slightly):

\begin{enumerate}
\item $H^1(G,k)=0$;
\item  $p$ does not divide $\dim V$;
\item  $H^1(G, V \otimes V^*)=0$; and
\item  $\EE(V)$ is spanned by the elements $\rho(g)$ with $\rho(g)$ semisimple.
\end{enumerate}

If $G$ is a finite group of order prime to $p$ (or $G$ is an algebraic or Lie
group in characteristic zero), then it is well known that $(G,V)$ is adequate.
In  this case, condition (4) is often referred to as Burnside's Lemma.  It is a trivial
consequence of the Artin-Wedderburn Theorem. 

These conditions are a weakening of the conditions used by Wiles and Taylor 
in studying the automorphic lifts of certain Galois representations.    Thorne \cite{T} 
generalized various results assuming
these hypotheses.  We refer the reader to \cite{T} for more references and details.

In particular, it was shown  in \cite[Theorem 9]{GHTT} that:

\begin{theorem} \label{adequate}  Let $k$ be a field of characteristic $p$ and
$G$ a finite group.  Let $V$ be an absolutely irreducible faithful $kG$-module.   If  $p  \ge 2 \dim V +2$, 
then $(G,V)$ is adequate. 
\end{theorem}

The proof depends on the classification of finite simple groups.  The main ingredients
include a result of the author \cite{Gu} that reduces to the problem to the case that
the subgroup of $G$ generated by elements of order $p$ is a central product
of quasisimple  finite groups of Lie type in characteristic $p$, a result of Serre \cite{S}
about complete reducibility of tensor products and results on the representation
theory of the groups of Lie type in the natural characteristic \cite{J}. 

In this note,  we consider (4) and show that this holds under some conditions (none
of these results depend upon the classification of finite simple groups). 
We say that $(G,V)$ is {\it weakly adequate}  if (4) holds.

Recall that a finite group is called $p$-solvable if every composition factor
of $G$ either has order $p$ or order prime to $p$.
It is known (cf. \cite[Theorem B]{I}) that if $G$ is $p$-solvable and $V$ is an absolutely
irreducible  $G$-module in characteristic $p$, then $G$ contains an absolutely irreducible
$p'$-subgroup, whence Burnside's Lemma immediately implies:

\begin{theorem} \label{p-solvable}  Let $G$ be a $p$-solvable subgroup, 
$k$ a field of characteristic $p$ and $V$ an absolutely irreducible $kG$-module.
If $p$ does not divide $\dim V$, then $(G,V)$ is weakly adequate.
\end{theorem}

This allows us to answer in the affirmative a question of R. Taylor for 
$p$-solvable groups.  

\begin{cor}  \label{taylor p-solvable} Let $G$ be a $p$-solvable subgroup, 
$k$ a field of characteristic $p$ and $V$ an absolutely irreducible $kG$-module.
If $(G,V)$ satisfies conditions (1), (2) and (3) above, then $(G,V)$ is adequate.
\end{cor} 

Recall that a $kG$-module $V$ is called primitive if $G$ preserves no nontrivial
direct sum decomposition of $V$.  

We can also show:

\begin{theorem}  \label{primitive}   Let $G$ be a $p$-solvable subgroup, 
$k$ a field of characteristic $p$ and $V$ an absolutely irreducible $kG$-module.
If  $V$ is primitive, then $(G,V)$ is weakly adequate.
\end{theorem}

Note that if $\dim V$ is a multiple of $p$, then no $p'$-subgroup can act irreducibly. 
We also can obtain some results for possibly infinite groups. 

\begin{theorem}   \label{infinite}  Let $k$ be algebraically closed of characteristic $p$.
Let $V$ be finite dimensional over $k$.  Let $\Gamma$ be an  irreducible subgroup of $\GL(V)$ with
Zariski closure $G$.    Let $G^0$ be the connected component of $G$ and $\Gamma^0 =  G^0 \cap \Gamma$.
Assume that either:
\begin{enumerate}
\item $[\Gamma:\Gamma^0]$ is not a multiple of $p$; or
\item  $\dim V$ is not a multiple of $p$ and $G/G^0$ is $p$-solvable.
\item  $V$ is primitive and $G/G^0$ is $p$-solvable.
\end{enumerate} 
Then $(G,V)$ is weakly adequate.
\end{theorem}

The only condition that is difficult to check for adequacy in the previous results is Condition (3). 
We do improve Theorem \ref{adequate} for $p$-solvable groups.
We first observe a result from \cite{Gu}.

\begin{theorem} \label{outer}  Let $k$ be a   field of characteristic $p$.   
Let $G$ be a finite subgroup of $\GL_n(k) = \GL(V)$.  Assume that $V$
is a completely reducible $kG$-module.   If $p > n$ and is not a Fermat
prime  or   $p > n+1$,
then $G$ has no composition factors of order $p$.
\end{theorem}

It is not difficult to extend this to the case of Zariski closed subgroups.  Also, the complete
reducibility hypothesis can be relaxed -- all we need to assume is that $G$ has no nontrivial
normal subgroup consisting of unipotent elements. This result is not explicitly stated in \cite{Gu} there but it is proved
there.  The result does depend upon the classification of finite simple groups (however, for 
$p$-solvable groups, it does not).  

It now easily follows that if $G$ is $p$-solvable and $V$ is a completely reducible $kG$-module
of small dimension, then $G$ is in fact a $p'$-group and this gives:

\begin{theorem} \label{adequate2}  Let $k$ be an algebraically closed field of characteristic
$p$.  Let $G$ be a $p$-solvable group.  Let $V$ be an irreducible 
$kG$-module.   Then $(G,V)$ is adequate if:
\begin{enumerate}
\item $p > \dim V$ with $p$ not  a Fermat prime; or
\item $p > \dim V + 1$.
\end{enumerate}
\end{theorem}

On the other hand, we present an infinite family of examples
of imprimitive absolutely irreducible $G$-modules  in characteristic
 $p$  with  $\dim V$ a multiple of $p$
(including cases where $G$ is $p$-solvable) wiht $(G,V)$  not weakly adequate.  
These are generalizations of examples of  Capdeboscq  and Guralnick.   

In order for this construction to give such examples where $p$ does not divide $\dim V$,
we were led to prove the following result in \cite{GG}:

\begin{theorem} \label{jgt}   
Let $p$ be a prime.  There exists a finite simple
group $G$ with a nontrivial Sylow $p$-subgroup $P$ such that some
coset of $P$ contains no $p'$-elements.
\end{theorem}

Thompson \cite{JGT} verified  this for $p=2$ in response to a question
of Paige.   

Using a variation of the Theorem \ref{jgt}, we show that for any prime $p$,
Taylor's question fails  (i.e. (1), (2) and (3) do not  necessarily imply (4)).

Note that another way to produce examples with $(G,V)$ not weakly adequate
is to find absolutely irreducible $G$-modules in characteristic $p$ such
that $(\dim V)^2$ is larger than the number of $p'$-elements in $G$. 
These examples are not so easy come by.   The only primitive example we know
is with 
 $G = {^2}F_4(2)'$ (the Tits group) and $V$ the  irreducible module  of dimension
 $2048$   in characteristic $2$.    The number of elements of odd
 order in $G$ is  $3,290,625 < (2048)^2$.  So  $(G,V)$ is not weakly adequate. 
 It is easy to see that $V$ is a primitive module (since $G$ contains no proper subgroups
 of index dividing $2048$).  
 
  This suggests the following variant of the problem:

\begin{question}  \label{Q2}  Let $G$ be a quasisimple finite group and $p$ a prime.
Classify all absolutely irreducible $G$-modules in characteristic $p$ such
that the number of $p'$-elements in $G$ is less than $(\dim V)^2$.
\end{question}

In particular, $(G,V)$ cannot be weakly adequate.   We suspect that there are very few such examples. 
 
The paper is organized as follows.  In the next section, we discuss $p$-solvable groups
and prove Theorems \ref{p-solvable} and \ref{primitive}. 
In the following  sections we prove  Theorem \ref{infinite} and Theorems \ref{outer} and \ref{adequate2}.  
In the last section, we consider necessary
conditions for induced modules to be weakly adequate.  This allows us to construct
many examples that are not weakly adequate including some whose dimension
is not a multiple of the characteristic.   In particular, this allows us to give a negative 
answer to Taylor's question. 

\section{$p$-solvable Groups}

We   prove Theorems  \ref{p-solvable} and \ref{primitive}.   As noted above,
the first result  follows by \cite[Theorem B]{I} (see also \cite{fong}).  We sketch
an elementary proof of a slight generalization of what we require.

We first prove a lemma about tensor products.  The first statement is well known.

\begin{lemma} \label{tensor}  Let $G$ be a group with a normal subgroup $N$.
Let $k$ be an algebraically closed field.  Let $V = U \otimes_k W$ be a 
finite dimensional    $kG$-module
where $U$ and $W$ are  irreducible $kG$-modules.  Assume that $N$ acts irreducibly on 
$U$ and trivially on $W$.  
\begin{enumerate}
\item  $V$ is an irreducible $kG$-module; and
\item  If $N$ consists of semisimple elements  and $(G,W)$ is weakly adequate,
then $(G,V)$ is weakly adequate.
\end{enumerate}
\end{lemma} 

\begin{proof}  We prove both statements simultaneously.   By assumption, 
$\EE(U) \otimes kI $ is the linear span of the images of $N$ in $\GL(U) \otimes kI$.

Since $W$ is $kG$-irreducible,  we can choose elements $g_i \in G$ such
that $g_i$ acts as $a_i \otimes b_i \in GL(U) \otimes GL(W)$ where the
$b_i$ form a basis for $\EE(W)$.  If $(G,W)$ is weakly adequate, we can furthermore
assume that the $g_i$ are semisimple elements. 

Thus, the images of the elements $Ng_i$ span $\EE(U) \otimes \EE(W) = \EE(V)$.
This shows that $V$ is an irreducible $kG$-module and that $(G,V)$ is weakly adequate
if $N$ consists of semisimple elments and the $g_i$ are semisimple 
(because then $Ng_i$ consists of semisimple elements).
\end{proof}

Note that  in (2) above,  $(G,V)$ weakly adequate implies that $(G,W)$ is weakly adequate. 

If $p$ is a prime dividing $|G|$, a subgroup $H$ is called a $p$-complement
if $p$ does not divide $|H|$ but $[G:H]$ is a power of $p$.   It is an easy
exercise to see  that the following holds (just choose a minimal normal subgroup and apply the
Schur-Zassenhaus result):

\begin{lemma}  Let  $G$ be a $p$-solvable group.  Any $p'$-subgroup of $G$ is contained
in a $p$-complement and all $p$-complements are conjugate.
\end{lemma}

We state the next result for irreducible groups rather than absolutely irreducible groups.  
Most results in the literature assume the latter.

\begin{lemma}  Let $G$ be a $p$-solvable group, $k$ a field of characteristic $p$
and $V$ an irreducible $kG$-module.  Let $F = \EE_G(V)$.  Assume that $p$ does not divide
$\dim_F V$.  Then a $p'$-complement $H$ of $G$ acts irreducibly on $V$ and 
$F = \EE_H(V)$.
\end{lemma}

\begin{proof}  First suppose that $k=F$ (i.e. $V$ is absolutely irreducible).
So we may assume that $k$ is algebraically closed.
We may also assume that $O_p(G)=1$ (since this acts trivially on $V$).  
Let $N$ be a minimal normal subgroup of $G$.  So $N$ is a $p'$-group.
First suppose that $N$ does not act homogeneously on $V$ (i.e.  $N$ has
at least two nonisomorphic simple submodules on $V$).  Then we can write
$V = \oplus_{i=1}^t  V_i$, where the $V_i$ are the homogeneous components of
$N$.  Let $S$ be the stabilizer of $V_1$.  Since $\dim V = t \dim V_1$,  $p$ is prime to
both $\dim V_1$ and $t$.   Let $K$ be a $p$-complement in $S$ and $H \ge K$
a $p$-complement of $G$.  By induction, $K$ is irreducible on $V_1$.  Since
$G =SH$ (since $[G:S] =t$ is prime to $p$), $H$ acts transitively on the set of $V_i$. 

Let $W$ be a nonzero $H$-submodule of $V$.  Since $N \le H$,  
$W  = \oplus (W \cap V_i)$ and since $H$ is transitive on the $V_i$,
we see that $W \cap V_1 \ne 0$.  Since $K$ acts irreducibly on $V_1$,
$V_1 \le W$ and since $H$ is transitive on the $V_i$,  $W=V$, whence the result.

Suppose that $N$ acts homogeneously.  It follows (cf.  \cite[Theorem 51.7]{CR}) that
(passing to a $p'$-central cover if necessary),  $V  \cong U \otimes_k W$
where $U, W$ are irreducible $kG$-modules with $N$ irreducible on $U$
and trivial on $W$.  If $H$ is a $p$-complement, then  by induction, 
$U$ and $W$ are irreducible $kH$-modules.   By Lemma \ref{tensor}, this implies
that $H$ acts irreducibly on $V$.

Now suppose that $k$ is not $F$.  Since $G$ is finite, we can assume
that $k$ is a finite field.  We can view $V$ as an absolutely irreducible $FG$-module,
  By the proof 
above,  $V$ is absolutely irreducible as an $FH$-module.  Thus,
$F =\EE_H(V)$.  Since $V$ is a semisimple $kH$-module (by Maschke's
theorem) with endomorphism ring a field, $V$ is an irreducible $kH$-module.
\end{proof}

Of course, if $p$ does divide $\dim_F  V$, then $V$ cannot possibly be
irreducible  restricted to  $H$, since the dimension of any absolutely irreducible
$H$-module in characteristic $p$ divides $|H|$.  Isaacs \cite{I} proves
much more than we do above and in particular studies the restriction
of $V$ to $H$ in all cases.   These ideas are related to the Fong-Swan
theorem:  every absolutely irreducible $G$-module is the reduction
of a characteristic zero module.

Theorem \ref{p-solvable} now follows by Burnside's Lemma.
Theorem \ref{primitive} now follows from the following observation:

\begin{lemma}  Let $G$ be a $p$-solvable group with $k$ algebraically
closed of characteristic $p$. If $V$ is a primitive $kG$-module,
then $p$ does not divide $\dim V$.
\end{lemma}

\begin{proof}   As above, we may assume that $O_p(G)=1$.  Let
$N$ be a minimal normal noncentral subgroup of $G$.  
Then $N$ is a $p'$-group and acts homogeneously on $V$.
If $N$ acts irreducibly, then $\dim V$ divides $|N|$ and the result
holds.  Otherwise, $V = U \otimes_k W$ where $U$ and $W$ are
primitive $kG$-modules, whence the result follows by induction
on dimension.
\end{proof}

We now give an example to show that  conditions (1), (2) and (4) do not
guarantee that condition (3) holds (even for solvable groups).

Let $r \ne p$ be an odd prime.  Let $R$ be an extraspecial $r$-group
of exponent $r$ and order $r^{1+2a}$.   Let $s$ be a prime distinct
from $p$ and $r$.  Let $S$ be an $s$-group with a faithful absolutely
irreducible $\FF_pS$ module $W$.  Let $X$ be an irreducible $\FF_pS$-submodule
of the semisimple module $W \otimes W^*$.   Set  $K = XS$, a semidirect product.   We can choose $a$ sufficiently
large so that $K$ embeds in $\Sp(2a,r)$ and so $K$ acts as a group of automorphisms
of $R$.   Then $RK \le R\Sp(2a,r)$ has an irreducible module  $U$
over $k$ of dimension $p^a$.   Set $V = U \otimes_k W$  (where we extend scalars and
view $W$ over $k$).     Then $V \otimes V^* \cong (U \otimes U^*) \otimes (W \otimes W^*)$.
Note that $V = V^R \oplus [R,V]$. and $V^R \cong W \otimes W^*$.
Thus, $H^1(G,V \otimes V^*)= H^1(G/R, W \otimes W^*) \cong \Hom_S(X, W \otimes W^*) \ne 0$.

\section{Infinite Groups}

Let $k$ be an algebraically closed field of characteristic $p \ge 0$.
Let $\Gamma$ be an absolutely irreducible subgroup of $\GL_d(k) = \GL(V)$.
Let $G$ be the Zariski closure of $\Gamma$.  Let $G^0$ denote
the connected component of $1$ in $G$.  Set $\Gamma^0 = \Gamma \cap G^0$.
Note that $G = \Gamma G^0$, whence $G/G^0 \cong \Gamma/\Gamma^0$.

We first note:

\begin{lemma}  \label{ss}  Let $G$ be a reductive algebraic  group
over   $k$  
(i.e. $G^0$ is reductive).  Let $g_i \in G, 1 \le i \le r$ be  such that  the order of
$g_iG^0/G^0$ is not a multiple of $p$.     Then
$X:=\{ x \in G^0 | \  g_i x \  \text{is \ semisimple} \}$ contains a Zariski open dense subset
of $G^0$.
\end{lemma}

\begin{proof}  Since the intersection of finitely many open dense sets is open
and dense, it suffices to proves this for $r=1$.  A straightforward argument
reduces this to the case that $G^0$ is a simple algebraic group and $g_1$
is either inner or is in the coset of a graph automorphism.  If $g_1$ is inner,
the result follows since the set of regular semisimple elements is open and dense.
If $g_1$ is a graph automorphism, the same is true --  see \cite[Lemma 6.8]{GM}. 
\end{proof}

Applying this to the Zariski closure of $\Gamma$, we immediately obtain:

\begin{cor} \label{ss2}  Let  $g_i$ be a finite set of  elements of  $\Gamma$ such
that  none of  the orders of 
$g_i\Gamma^0$ in  $\Gamma/\Gamma^0$  are a multiple of $p$.  
Then   $X:=\{ x \in \Gamma^0 | \  g_i x \  \text{is \ semisimple} \}$
is Zariski dense in $G^0$.
\end{cor} 

In particular, this implies:

\begin{cor}  \label{char 0}  If $k$ is algebraically closed of characteristic $0$ and
$V$ is an irreducible finite dimensional $k\Gamma$-module, then
$(\Gamma, V)$ is weakly adequate.
\end{cor}

 \begin{lemma}  \label{inf-prim}   Suppose that $V = U \otimes_k W$ where $U$ and $W$
are irreducible  finite dimensional $k\Gamma$-modules and that $\Gamma^0$ acts irreducibly
on $U$ and trivially on $W$.  If $(\Gamma, W)$ is weakly adequate, then
$(\Gamma, V)$ is weakly adequate.
\end{lemma}

\begin{proof}  If $(\Gamma, W)$ is weakly adequate, then we can choose finitely many 
$g_i \in \Gamma$ semsimple with $g_i = a_i \otimes b_i \in \GL(U) \otimes GL(W)$
where the span of the $b_i$ is $\EE(W)$.   Let $X$ be the subset of $\Gamma^0$
consisting of all elements $x$ such that $g_ix$ is semisimple for all $g_i$ (take $g_1=1$).
By Corollary \ref{ss2},  $X$ is Zariski dense in $G^0$.   
Thus,  the linear span of $X$ is Zariski dense in the linear span of $G^0$
which is precisely $\EE(U) \otimes kI$.   Thus, $\cup g_iX$ consists of semisimple
elements and contains a basis for $\EE(V) = \EE(U) \otimes \EE(W)$.  
\end{proof} 

We now prove Theorem \ref{infinite}.

\begin{proof}
First suppose that $p$ does not divide $[\Gamma:\Gamma^0]$.   It follows
by  Corollary \ref{ss2}  that the set of semisimple elements of $\Gamma$
contain a Zariski dense  subset of $G$.  Thus, the linear span
of the semisimple elements of $\Gamma$ is Zariski dense in the linear span
of $G$.  Since linear spaces are closed, it follows that the two sets
have the same linear span, whence the result.   

Next suppose that $p$ does not divide $d$ and $G/G^0$ is $p$-solvable.
Let $H/G^0$ be a $p$-complement in $G/G^0$.  The exact same proof
as in the previous section shows that $H$ is irreducible on $V$.  
Thus, $\Gamma \cap H$ (which is Zariski dense in $H$) is also irreducible
on $V$.   Now apply (1) to $\Gamma \cap H$.    

Finally consider (3).   Since $V$ is primitive, $\Gamma^0$ acts homogeneously on $V$.
Thus, $V = U \otimes_k W$, where $U$ and $W$ are irreducible $k\Gamma$-modules,
$\Gamma$ acts irreducibly on $U$ and trivially on $W$.   Since $\Gamma/\Gamma^0$ is $p$-solvable and $\Gamma^0$
is trivial on $W$,  $(\Gamma, W)$ is weakly adequate by Theorem \ref{primitive}.  Now apply 
Lemma \ref{inf-prim}.  
\end{proof} 

\section{Composition Factors}

We first prove Theorem \ref{outer}.  As we noted this is essentially in \cite{Gu}.
We sketch the proof indicating in particular how the classification is not required
for the case of $p$-solvable groups.  

\begin{theorem}  Let $G$ be a completely reducible finite subgroup of $\GL_n(k)=\GL(V)$
with $k$ a field of characteristic $p$.  If $H^1(G,k) \ne 0$, then
either $n \ge p$ or $p$ is a Fermat prime and $n = p-1$. 
\end{theorem}   

\begin{proof}  If $p \le 3$, then all we are asserting is that $n \ge 2$ and
the result is clear.  So assume that $p \ge 5$ and $p > n$ with $H^1(G,k) \ne 0$. 

Let $N$ be the normal subgroup generated by elements of order $p$.  
Then $H^1G,k)$ embeds into $H^1(N,k)$ and so we may assume that 
$N=G$.   Let  $A$ be a minimal normal noncentral subgroup of $G$.
We consider four cases:\\
 
Case 1.   $A$ is an elementary abelian $r$-group for some prime $r \ne p$.\\
 
 Then $G$ permutes the weight spaces of $A$ and since $G$ is generated by
 elements of order $p$, some element of order $p$ does not centralize $A$, whence
 it must have an orbit of size $p$ and so $n \ge p$.\\
 
 Case 2.   $A$ is of symplectic type  (i.e $A/Z(A)$ is elementary abelian of order
 $r^{2a}$ for some prime $r \ne p $ with $Z(A)$ of order $r$ if $r$ is odd or of order
 $2$ or $4$ if $r=2$;  moreover, $A$ has exponent $r$ is $r$ is odd and
 has exponent $4$ if $r=2$).   \\
 
 Again, some element $g$ of order $p$ acts nontrivial on $A$.  Thus,
 $g$ embeds in $\Sp(2a,r)$, whence $p \le r^a + 1$ with equality if and only
 $r =2$ and $p$ is a Fermat prime.  Since the minimal
 faithful representation of $A$ in characteristic $p$ is $r^a$, the result follows. 
 
 Case 3.  $A$ is a central quotient of a direct product of quaisisimple subgroups
 and $p$ does not divide $|A|$.  \\
 
 Again some element $g$ of order $p$ acts nontrivially on $A$.  If $g$ does not preserve
 each quasisimple factor of $A$, then there are least $p$ such factors and we easily 
 see that $n \ge 2p$.   So $g$ normalizes each factor of $A$.  Thus, $A$ is quasisimple.
 We can assume that $A$ acts homogeneously (and nontrivially) on $V$ (otherwise,
 we may assume that $g$ permutes the homogeneous factors and so there would
 be at least $p$ of them, whence $n \ge 2p$).     Since
 $p$ does not divide $|A|$,  it follows by Sylow's theorem, that $g$ will normalize
 a Sylow $r$-subgroup of $A$ for each prime $r$ dividing $|A|$.   Thus,
 $g$ will act nontrivially on some Sylow $r$-subgroup of $G$ and the result follows
 from cases 1 and 2.\\
 
 Case 4.  $A$ is a central quotient of a direct product of quaisisimple subgroups
 and $p$ does   divide $|A|$.  \\
 
 Unfortunately, we do not have a proof without the classification (although we suspect
 there is one).     We argue as in case 3.  Now apply \cite[Theorem B]{Gu} to conclude
 that $A$ is of Lie type in characteristic $p$.  It follows that $g$ must induce a field
 automorphism and this forces $n \ge 2p$ (one further possibility is that $A = J_1$
 with $p=11$, but then $A$ has no outer automorphisms of odd order).  
 \end{proof}
 
 Now Theorem \ref{outer} follows immediately (if there is a composition factor
 of order $p$, there will be a normal subgroup $N$ of $G$ with $H^1(N,k) \ne 0$
 and $N$ is still completely reducible). 
 
 An immediate corollary is:
 
 \begin{cor}  \label{strong}  Let $G$ be a completely reducible $p$-solvable subgroup of $\GL_n(k)=\GL(V)$
with $k$ a field of characteristic $p$.  If $p$ divides $|G|$, then 
either $n \ge p $ or $n = p-1 $ with $p$ a Fermat prime.
\end{cor}

 We now prove  Theorem \ref{adequate2}.

 \begin{proof}  Assume that $p > \dim V$ (or $p > \dim V + 1$ if $p$ is a Fermat prime)
 and that $G$ is an irreducible subgroup of $\GL(V)$ as in the hypotheses.
 By the corollary $G$ is in fact a $p'$-group, whence $(G,V)$ is adequate. 
 \end{proof}

\section{Induced Modules}

Let $k$ be an algebraically closed field of characteristic $p >  0$.  
Suppose that $V = \mathrm{Ind}_K^G (W)$.   Let $g_i$ be a set of
coset representatives for the cosets of $K$ in $G$.   So
we can write $V = W_1 \oplus \ldots \oplus  W_m$ where $m=[G:K]$
and  $W_i = g_i \otimes W$.  

So $\EE(V) = \oplus_{ij} \Hom(W_i, W_j)$.   Let $\pi_{ij}$ be the corresponding
projection from $\EE(V)$ to $\Hom(W_i, W_j)$.   
Note that the set of $g \in G$ such that $\pi_{1j}(g) \ne 0$ is $g_jK$.
This observation yields: 

\begin{lemma} If $(G,V)$ is weakly adequate, then $\pi_{1j}$ maps
the set of $p'$-elements of $g_jK$ to a spanning set of $\Hom(W_1,W_j)$.
In particular, if some coset $g_jK$ contains no $p'$-elements, then
$(G,V)$ is not weakly adequate.
\end{lemma}

Using this criterion, we can produce many  examples $(G,V)$ which are not 
weakly adequate.  Of course, we want
$V$ to be irreducible and we also want $G$ to be generated by $p'$-elements. 

Here is our first family of examples. 

Let $H$ be any finite group whose order is divisible by $p$ with $H$ generated
by its $p'$-elements.  Let $r$ be a prime not equal to $p$ and let 
$A$ be an irreducible $H$-module such that $H$ has a regular orbit
on  $\Hom(A,k^*)$  (this can be easily arranged - if $r$ is sufficiently large, then
any faithful irreducible module $A$ will have this property).   Set $G=AH$, a
semidirect product. 

Let $W$ be a $1$-dimensional $kA$-module with character $\lambda \in \Hom(A,k^*)$ 
so that $\lambda$ is in a regular $G$-orbit.
Set $V = W_A^G$.  We note that  $V$ is an irreducible $kG$-module of dimension equal to $|H|$
(since $V$ is a direct sum of $1$-dimensional non-isomorphic $kA$-modules permuted
transitively by $H$).  Clearly, $G$ is generated by its $p'$-elements.   If $g \in G$ has order
divisible by $p$ the coset $gA$ has no $p'$-elements, whence:

\begin{theorem} \label{example1}  $(G,V)$ is not weakly adequate.
\end{theorem}

In particular, we can take $G = AS_3$ where $A$ is elementary abelian of order $25$
with $p=3$ and $\dim V = 6$.    

In fact, we can generalize these examples.   Here is the setup:
\begin{enumerate}
\item   Let $L$ and $T$ be  finite groups each generated by  $p'$-elements.
\item   Let  $W$ be an absolutely  irreducible faithful $kL$-module. 
\item   Let $T_1$ be a subgroup of $T$ of index $t$ such that $T_1$ contains no nontrivial 
normal subgroup of $T$
and such that some coset $xT_1$ of $T_1$ in $T$ contains no $p'$-elements (eg,  if $T_1$
is a proper subgroup of a Sylow $p$-subgroup $P$ of $T$, then let $x \in P \setminus{T_1}$).
\end{enumerate}

Set $G= L \wr T = NT$, where $N=L_1 \times \ldots \times L_m$ with $L_i \cong L$ and $m=[T_1:T]$
Then $G$ acts on $V:= W_1 \oplus \ldots \oplus W_t$ where $W_i \cong W$  ($L_i$ acts as $L$
on $W_i$ and trivially on $W_j$ with $j \ne i$ and $T$ permutes the $W_i$ as it does the coset of $T_1$).  
 We can also describe $V$ as $\mathrm{Ind}_K^G (W_1)$ where $K=NT_1$ with $L_1$ acting on $W_1$
 as $L$ does on $W$ and $(L_2 \times \ldots L_t)T_1$ acting trivially on $W_1$).   
 
 \begin{theorem} \label{example2}  With notation as above, $V$ is  a faithful  irreducible $kG$-module of dimension
 equal to $m \dim W$, $G$ is generated by $p'$-elements and
  $(G,V)$ is not weakly adequate.
 \end{theorem}
 
 \begin{proof}   Since the $W_i$ are nonisomorphic irreducible $kN$-modules and $T$ permutes
 them transitively, $V$ is irreducible.   Since $L$ and $T$ are generated by $p'$-elements, so is $G$. 
 Since $T_1$ contains no nontrivial normal subgroup of $T$, the kernel of this representation
 would be contained in $N$.   Clearly $N$ acts faithfully. 
 Since the coset  $xT_1N$ contains no $p'$-elements, the result follows.
 \end{proof}
 
 Using the result of \cite{GG} for any odd prime $p$, we can find
 a sufficiently large $q$ with $p$ exactly dividing $q-1$ so that  for  $T = \PSL_2(q)$  
 and $T_1$ a dihedral subgroup of order $2p$, we can  find a $t \in T$ with $tT_1$ containing 
 no $p'$-elements.
 
 This allows us to give a negative answer to Richard Taylor's question.

  \begin{theorem}  \label{taylor question} Let $k$ be an algebraically closed field of characteristic  $p$.  Let $T = \PSL_2(q)$ and 
  let $T_1$ be a subgroup of $T$ isomorphic to a dihedral group of order $2p$ as above. 
  Let $L$ be a cyclic group order $2$ and let $W$ be the nontrivial $1$-dimensional
  $kL$-module.  Set $G = L \wr_{T_1} T$, $N=L \times \ldots \times L$ of order $2^m$ with $m=[T_1:T]$.  
  Let $T_1$ act trivially on $W$.  
  Set $V = \mathrm{Ind}_{K}^G (W)$ where $K=NT_1$  Then
  \begin{enumerate}
  \item   $V$ is an absolutely irreducible $kG$-module of dimension $m$ (and so prime to $p$); 
  \item   $G$ satisfies conditions (1), (2) and (3) of the introduction;  and
  \item    $G$ is not adequate.
  \end{enumerate}
  \end{theorem}
  
  \begin{proof}  As we have seen above,  the first condition holds and $(G,V)$ is not weakly adequate by
  the construction   Clearly $p$ does not divide $m =\dim V$.  By construction $G$
  is generated by $p'$-elements.   So it remains to that show $H^1(G, V \otimes V^*)=0$.
  
  Set $U = V \otimes V^*$.  Since $N$ is a normal $p'$-group, it follows that
  $U = C_U(N) \oplus [N,U]$ where $C_U(N)$ are the fixed points of $N$ on $U$ and $[N,U]$
  is the submodule generated by all nontrivial irreducible $N$-submodules.  By the inflation restriction sequence,
  it follows that $H^1(G,[N,U])=0$.   Note that $ \dim C_U(N) = m$ and indeed $C_U(N)$ contains $U_1:=W \otimes W^*$
  and the stabilizer of $U_1$ in $G$ is $NT_1$.  Thus,  $C_U(N) \cong \mathrm{Ind}_K^G (k)$.
  So by Shapiro's Lemma,  $H^1(G, C_U(N)) \cong H^1(K, k) \cong H^1(T_1, k) =0$.
  \end{proof} 
  
  One can also produce examples showing that Taylor's question has a negative answer with $p=2$ as well.
  For example, we can take $T=\PSL_2(137)$   and $T_1 =  A_4  \le T$ and $L$ cyclic of order 
  $3$ with $W$ a $1$-dimensional nontrivial $L$-module. 
   
  Here is a variation of Taylor's question:
  
  \begin{question} \label{taylorq2}  Let $V$ be an absolutely irreducible primitive $kG$-module.  If
  $(G,V)$ satisfies (1), (2) and (3) of the introduction,  is $(G,V)$ adequate?
  \end{question}

  Now suppose that $G$ is $p$-solvable.  Let $V$ be an irreducible $kG$-module.  If $N$ is a noncentral
  normal $p'$-subgroup of $G$ that acts homogeneously, then as usual we can write $V = U \otimes_k W$.
  By Lemma \ref{tensor} and the remark following it,  $(G,V)$ is weakly adequate if and only if $(G/N, W)$ is.
  Thus, if this is the case, the problem reduces to a smaller module.  So we may assume that no 
  noncentral normal $p'$-subgroup acts homogeneously.   In this case, set $N=O_{p'}(G)$ (the largest normal
  $p'$-subgroup).   Then $V = V_1 \oplus \ldots \oplus V_m$ with $m > 1$ where the $V_i$ are the $kN$-components
  of $V$.  Thus,  $V = \mathrm{Ind}_K^G (V_1)$ where $N \le K$.   We ask:

  \begin{question} \label{Q4}    If $G$ is $p$-solvable and every coset $gK$ of $K$ contains a semisimple element, is $(G,V)$
  weakly adequate?
  \end{question}
 
 If the answer is yes, then we have an essentially complete answer as to when an absolutely irreducible
 $kG$-module $V$ is weakly adequate for $G$ a $p$-solvable group.
   
\section{Acknowledgments}

 The  author was partially supported by the NSF
 grant DMS-1001962.  Some of this work was carried out at the  Institute for Advanced Study.
 The author would like to thank Inna Capdeboscq for her help  on constructing the first examples
 which were not weakly adequate, John Thompson for pointing out his article \cite{JGT}, Simon Guest  and
 Daniel Goldstein for help with some MAGMA computations and Richard Taylor for many helpful  comments 
 and questions.   Finally, we would like to thank the referee for their careful reading and helpful comments.

\end{document}